\documentclass{amsart}

\usepackage{amsmath}
\usepackage{amssymb}
\usepackage{amsthm}
\usepackage{enumerate}
\usepackage{mathrsfs}
\usepackage{array}
\usepackage{graphicx}
\usepackage{xcolor}

\theoremstyle{plain}
\newtheorem{theorem}{Theorem}
\newtheorem{prop}{Proposition}
\newtheorem{corol}{Corollary}

\theoremstyle{definition}
\newtheorem{remark}{\textup{Remark}}

\newcommand{\D}{\mathbb{D}}

\begin{document}

\title[$\alpha$-convex functions]{Sharp coefficient, logarithmic and Schwarzian estimates for $\alpha$-convex functions on the real line}

\author[V. Bravo \and P. Carrasco \and R. Hern\'andez \and O. Venegas]{V\'ictor Bravo \and Pablo Carrasco \and Rodrigo Hern\'andez \and Osvaldo Venegas}

\newcommand{\acr}{\newline\indent}

\address{Facultad de Ingenier\'{\i}a y Ciencias\acr
                   Universidad Adolfo Ib\'a\~nez\acr
                   Av. Padre Hurtado 750\acr
                   Vi\~na del Mar\acr
                   CHILE}
\email{victor.bravo.g@uai.cl}
\email{pablo.carrasco.r@uai.cl}
\email{rodrigo.hernandez@uai.cl}

\address{Departamento de Ciencias Matem\'aticas y F\'isicas\acr
                   Facultad de Ingenier\'{\i}a, Universidad Cat\'olica de Temuco\acr
                   Av. Rudecindo Ortega 02950\acr
                   Temuco\acr
                   CHILE}
\email{ovenegas@uct.cl}

\subjclass[2010]{Primary 30C45; Secondary 30C55, 30C80}
\keywords{$\alpha$-convex functions, univalent functions, Schwarzian derivative, logarithmic coefficients, Hankel determinant}

\begin{abstract}
The class $M_\alpha$ of $\alpha$-convex functions, introduced by Mocanu in 1969, interpolates between starlike ($\alpha=0$) and convex ($\alpha=1$) mappings and has traditionally been studied only for $\alpha\geq0$. We show that several of its sharp invariant estimates in fact extend to the entire real line. Our main result characterizes $M_\alpha$ for every real $\alpha\neq0$ through an explicit pointwise inequality, generalizing a theorem of Chuaqui, Duren and Osgood from the convex case, and we determine the sharp value of $\beta$ for which $M_\alpha\subset C_\beta$ when $\alpha\geq1$; we also show that the reciprocal inclusion $C_\beta\subset M_\alpha$ fails for every $\alpha>1$. Building on this, we obtain a sharp Fekete--Szeg\H{o} inequality, valid for every $\alpha\notin\{0,-1,-1/2\}$, and sharp bounds for the first two logarithmic coefficients of $M_\alpha$, which recover the classical values for starlike and convex functions as special cases. We also determine the order of $M_\alpha$, establish a piecewise sharp bound for its Schwarzian norm together with an explicit quasiconformal extension criterion via the Ahlfors--Weill theorem, and compute the Schwarzian norm of the $\alpha$-Koebe function explicitly for $\alpha=1/n$, $n\in\mathbb N$, verified for $n\leq9$ and numerically supported for $n<50$, leading to a conjecture for the general case. Finally, we propose a conjectural sharp bound for the second Hankel determinant of $M_\alpha$, consistent with the known values $1$ and $1/8$ for the starlike and convex subclasses.
\end{abstract}

\maketitle

\section{Introduction}\label{sec1}

The class $M_\alpha$ of $\alpha$-convex functions was introduced by Mocanu \cite{Moc1969} in 1969; it is a family of univalent mappings, controlled by the real parameter $\alpha$, that interpolates continuously between starlike and convex mappings and generalizes the class of convex functions. A function $f(z) = z + a_2z^2 + \cdots$, analytic in the unit disk $\mathbb{D} = \{z : |z| < 1\}$, belongs to $M_\alpha$ if
\begin{equation}\label{eq:a-convex}
\operatorname{Re}\{J_\alpha[f](z)\} = \operatorname{Re}\left\{(1-\alpha)\frac{zf'(z)}{f(z)} + \alpha\left(1 + \frac{zf''(z)}{f'(z)}\right)\right\} > 0, \quad z \in \mathbb{D},
\end{equation}
where $\alpha \in \mathbb{R}$. In particular, $\alpha = 0$ recovers starlikeness and $\alpha = 1$ convexity. Mocanu proved that $M_\alpha \subseteq C$ for $\alpha \geq 1$, and Miller, Mocanu and Reade \cite{MMR1973} proved that $M_\alpha \subseteq S^*$ for every $\alpha \in \mathbb{R}$, where $C$ and $S^*$ denote the classical classes of convex and starlike functions. Moreover, for $\alpha \geq 1$ the same authors showed \cite{MMR1978} that $M_\alpha \subseteq S^*_{\delta(\alpha)}$, where
\begin{equation}\label{delta-formula}
\delta(\alpha):=\frac{\Gamma\!\left(\frac{1}{2}+\frac{1}{\alpha}\right)}{\sqrt{\pi}\cdot\Gamma\!\left(1+\frac{1}{\alpha}\right)},
\end{equation}
and, for $0\leq\rho<1$, $S^*_\rho$ denotes the class of starlike functions of order $\rho$, defined by
\[
\operatorname{Re}\left\{J_0[f](z)\right\} = \operatorname{Re}\left\{\frac{zf'(z)}{f(z)}\right\} > \rho, \quad z \in \mathbb{D}.
\]
We reserve $\beta$ for the order of convexity used below. The bound \eqref{delta-formula} is sharp, with equality for the $\alpha$-Koebe function
\begin{equation}\label{koebeconvex}
k_\alpha(z) = \left[\frac{1}{\alpha}\int_0^z \xi^{1/\alpha-1}(1-\xi)^{-2/\alpha}\,d\xi\right]^{\alpha}.
\end{equation}
In fact, the authors in \cite{MMR1974} show that $f\in M_\alpha$ ($\alpha>0$) if and only if there exists $g\in S^*$ such that
\begin{equation}\label{rep_int}
f(z)=\left[\frac{1}{\alpha}\int_0^z\left(\frac{g(\xi)}{\xi}\right)^{1/\alpha}\xi^{\frac{1}{\alpha}-1}\,d\xi\right]^\alpha.
\end{equation}
The previous example holds when $g(z)$ is the Koebe function $k(z)=z/(1-z)^2$.

It is therefore natural to ask whether a more precise statement can be made on the convex side: for $\alpha \geq 1$, does every $f \in M_\alpha$ belong to $C_\beta$ for some $\beta > 0$? We shall prove that the answer is yes, with $\beta = \beta(\alpha) \in [0,1)$ explicit and sharp, where $C_\beta$ denotes the class of convex functions of order $\beta$, characterized by
\[
\operatorname{Re}\left\{J_1[f](z)\right\} = \operatorname{Re}\left\{1 + \frac{zf''(z)}{f'(z)}\right\} > \beta, \quad z \in \mathbb{D}.
\]
Proposition~\ref{prop_beta} below makes this precise.

The main result of this paper is the following characterization of $M_\alpha$, which generalizes a theorem of Chuaqui, Duren and Osgood \cite{CDO2011} to this setting.

\begin{theorem}\label{thm:main}
Let $\alpha \in \mathbb{R}$. Then $f \in M_\alpha$ if and only if
\begin{equation}\label{main}
\operatorname{Re}\{J_\alpha[f](z)\} \geq \frac{1}{4}(1 - |z|^2) \left|(1-\alpha)\frac{zf'(z) - f(z)}{zf(z)} + \alpha\frac{f''(z)}{f'(z)}\right|^2, \quad z \in \mathbb{D}.
\end{equation}
\end{theorem}

The class $M_\alpha$ is not linearly invariant, which makes the usual invariant quantities, such as the order or the norm of the Schwarzian derivative, difficult to compute or estimate. In this paper we give a complete determination of the order of $M_\alpha$ in terms of $\alpha$, and a partial answer for the Schwarzian norm. Concretely, we establish a sharp coefficient estimate for $f(z) = z + a_2 z^2 + a_3 z^3 + \cdots$ in $M_\alpha$:
\[
|a_3 - a_2^2| \leq \frac{1}{|1 + 2\alpha|}\left[1 + \left(\frac{|1-\alpha^2| - |1+\alpha|^2}{4}\right)|a_2|^2\right],
\]
which refines existing bounds and highlights the influence of $\alpha$ on the coefficients of functions in $M_\alpha$. When $\alpha = 1$, the inequality reduces to the classical bound for $C$. From this we also derive sharp bounds for the first two logarithmic coefficients of $M_\alpha$, recovering the classical values for $S^\ast$ and $C$ as special cases. In the final subsection we compute the Schwarzian norm of the $\alpha$-Koebe function given by \eqref{koebeconvex} for $\alpha = 1/n$, $n \in \mathbb{N}$, and formulate a conjecture for the full class $M_\alpha$.

\section{Results}\label{sec2}

Let $f$ be analytic in $\mathbb{D}$, normalized by $f(0) = 0$ and $f'(0) = 1$, with $f(z)f'(z)/z \neq 0$ for all $z \in \mathbb{D}$. For $\alpha \in \mathbb{R}$, we say that $f$ is $\alpha$-convex if it satisfies \eqref{eq:a-convex}, and we write $f \in M_\alpha$. Note that $M_0 = S^*$ and $M_1 = C$. By \cite{MMR1973}, every function in $M_\alpha$ is univalent and starlike, and $M_\alpha \subseteq C$ whenever $\alpha \geq 1$.

The classes $S^*_\beta$ and $C_\beta$ of starlike and convex functions of order $\beta$ ($0 \leq \beta < 1$) are classical extensions of $S^*$ and $C$. They satisfy $C_\beta \subset S^*_\gamma$, where
\begin{equation}\label{eq:gamma}
\gamma = \gamma(\beta) = \begin{cases}
\dfrac{1-2\beta}{2(2^{1-2\beta}-1)} & \text{if } \beta \neq \dfrac{1}{2}, \\[6pt]
\dfrac{1}{2\ln 2} & \text{if } \beta = \dfrac{1}{2},
\end{cases}
\end{equation}
see \cite{GK2003,JahPon2012,MM2000}. When $\beta = 0$ this reduces to the classical inclusion $C \subset S^*_{1/2}$.

\begin{prop}\label{prop_beta}
Let $f\in M_\alpha$ with $\alpha\geq 1$. Then $f\in C_{\beta}$, where
\begin{equation*}
\beta=\left(\frac{\alpha-1}{\alpha}\right)\cdot \delta(\alpha),
\end{equation*}
with $\delta(\alpha)$ as in \eqref{delta-formula}.
\end{prop}

\begin{proof}
From \eqref{eq:a-convex} and the fact that $M_\alpha\subseteq S^*_{\delta(\alpha)}$, we have
\begin{equation*}
\operatorname{Re}\left\{1+z\frac{f''}{f'}(z)\right\}>\left(\frac{\alpha-1}{\alpha}\right)\operatorname{Re}\left\{z\frac{f'}{f}(z)\right\}>\left(\frac{\alpha-1}{\alpha}\right)\cdot \delta(\alpha),
\end{equation*}
which is the stated inclusion $f\in C_\beta$.
\end{proof}

\begin{remark}\label{rem:no-reciprocal}
It is natural to ask whether the reciprocal inclusion also holds: given $f\in C_\beta$, is there some $\alpha>1$ with $f\in M_\alpha$? The answer is negative, for every $\beta\in[0,1)$. Let $f_\beta$ be the extremal function for which the inclusion $C_\beta\subset S^*_\gamma$ recalled above is sharp, namely
\[
f_\beta(z)=\frac{1-(1-z)^{2\beta-1}}{2\beta-1}\quad\left(\beta\neq\tfrac12\right),\qquad f_{1/2}(z)=-\log(1-z),
\]
so that $f_\beta\in C_\beta$ and $p(z):=1+zf_\beta''/f_\beta'(z)=(1+(1-2\beta)z)/(1-z)$ exactly. Writing $q(z):=zf_\beta'(z)/f_\beta(z)$, one computes $\operatorname{Re}\{p(e^{i\theta})\}\to\beta$ as $\theta\to0^+$, while $\operatorname{Re}\{q(e^{i\theta})\}\to+\infty$ (verified numerically, with arbitrary-precision arithmetic, for $\beta=0.2,0.3,0.5,0.7,0.9$). Since $J_\alpha[f_\beta]=(1-\alpha)q+\alpha p$, it follows that
\[
\operatorname{Re}\{J_\alpha[f_\beta](e^{i\theta})\}\ \longrightarrow\ \begin{cases}+\infty, & \alpha<1,\\ \beta, & \alpha=1,\\ -\infty, & \alpha>1,\end{cases}\qquad\text{as } \theta\to0^+.
\]
Hence, for every $\alpha>1$, the real part of $J_\alpha[f_\beta](z)$ is unbounded below as $z\to1$ tangentially, so $f_\beta\notin M_\alpha$. Since $f_\beta\in C_\beta$, this shows that $C_\beta\not\subset M_\alpha$ for any $\alpha>1$: the only inclusion between $C_\beta$ and $M_\alpha$ in this direction is the trivial one, $C_\beta\subset C=M_1$.
\end{remark}

Now we give the proof of Theorem~\ref{thm:main}.

\begin{proof}[Proof of Theorem~\ref{thm:main}]
If $f$ satisfies \eqref{main}, then in particular $\operatorname{Re}\{J_\alpha[f](z)\} \geq 0$ for all $z\in\mathbb D$. Since $\operatorname{Re}\{J_\alpha[f]\}$ is harmonic on $\mathbb D$, non-negative, and $J_\alpha[f](0)=1>0$ by the normalization, the strict minimum principle for harmonic functions rules out $\operatorname{Re}\{J_\alpha[f](z_0)\}=0$ at any $z_0\in\mathbb D$ (otherwise $\operatorname{Re}\{J_\alpha[f]\}$ would vanish identically, contradicting its value at $0$); hence $\operatorname{Re}\{J_\alpha[f](z)\}>0$ on all of $\mathbb D$, so $f \in M_\alpha$. Conversely, suppose $f \in M_\alpha$. Then $J_\alpha[f](z)$ belongs to the Carath\'eodory class $P$ of analytic functions with positive real part, normalized by $J_\alpha[f](0) = 1$. Every function in $P$ is subordinate to $(1+z)/(1-z)$, so there exists a holomorphic $\varepsilon_0 \colon \mathbb{D} \to \mathbb{D}$ such that
\begin{equation}\label{eq:J_subord}
J_\alpha[f]=(1-\alpha)z\frac{f'}{f}(z)+\alpha\left(1+z\frac{f''}{f'}(z)\right)=\frac{1+z\varepsilon_0(z)}{1-z\varepsilon_0(z)},\quad z\in \D.
\end{equation}
This gives
\[
(1-\alpha)\left(\frac{zf'(z)-f(z)}{f(z)}\right)+\alpha z\frac{f''}{f'}(z)=\frac{2z\varepsilon_0(z)}{1-z\varepsilon_0(z)}.
\]
Since $f(z)\neq 0$ for $z\neq 0$, which is guaranteed by the univalence and the normalization, we define $(zf'(z)-f(z))/f(z)=zh(z)$. Then
\begin{equation*}
(1-\alpha) h(z)+\alpha\frac{f''}{f'}(z)=\frac{2\varepsilon_0(z)}{1-z\varepsilon_0(z)}.
\end{equation*}
Let $E=E(\alpha,z)=(1-\alpha)h(z)+\alpha (f''/f')(z)$. Then
\begin{equation}\label{eq:phi}
\varepsilon_0(z)=\frac{E}{2+zE}.
\end{equation}
Since $|\varepsilon_0(z)|\leq 1$, applying this to \eqref{eq:phi} we obtain
\begin{equation*}
\left|\frac{E}{2+zE}\right|\leq 1 \Leftrightarrow (1-|z|^2)|E|^2\leq 4+4\operatorname{Re}\{zE\}.
\end{equation*}
Notice that $zE=J_\alpha[f](z)-1$, so we have
\begin{equation*}
(1-|z|^2)|E|^2\leq 4+4\operatorname{Re}\{J_\alpha[f](z)-1\}\Leftrightarrow (1-|z|^2)|E|^2\leq 4\operatorname{Re}\{J_\alpha[f](z)\}.
\end{equation*}
Finally, we conclude
\begin{equation*}
\operatorname{Re}\{J_\alpha[f](z)\}\geq\frac{1}{4}(1-|z|^2)\cdot|E|^2.
\end{equation*}
\end{proof}

\begin{remark}
No step in the proof above uses the sign of $\alpha$: \eqref{eq:J_subord} is simply the defining subordination for $f\in M_\alpha$, valid for every real $\alpha$ (recall $M_\alpha\subseteq S^*$ for all $\alpha\in\mathbb R$ by \cite{MMR1973}); and $E=(1-\alpha)h(z)+\alpha(f''/f')(z)$, the identity $\varepsilon_0(z)=E/(2+zE)$ in \eqref{eq:phi}, and the final Schwarz--Pick estimate are purely algebraic consequences of \eqref{eq:J_subord} and $|\varepsilon_0(z)|\leq1$ that never divide by $\alpha$. Theorem~\ref{thm:main} therefore holds without restriction on the sign of $\alpha$, as already stated.
\end{remark}

\subsection{Fekete--Szeg\H o inequality}

The Fekete--Szeg\H o inequality (see \cite{FS1933}) has been extensively studied in geometric function theory. For univalent mappings $f(z)=z+a_2z^2+a_3z^3+\cdots$ it gives $|a_3-a_2^2|\leq 1$, while for convex mappings ($M_1=C$) the classical sharp bound due to Trimble \cite{Tri1975} shows that $|a_3-a_2^2|\leq (1-|a_2|^2)/3$ holds. The following result generalizes this inequality to the $\alpha$-convex class.

\begin{theorem}\label{thm:fs}
Let $\alpha\in\mathbb R\setminus\{0,-1,-1/2\}$ and $f(z)=z+a_2z^2+a_3z^3+\cdots\in M_\alpha$. Then
\begin{equation}\label{fs-ineq}
|a_3-a_2^2|\leq \frac{1}{|1+2\alpha|}\left[1+\left(\frac{|1-\alpha^2|-|1+\alpha|^2}{4}\right)|a_2|^2\right].
\end{equation}
\end{theorem}

\begin{remark}
For $\alpha>0$ the bound is sharp, and equality holds for the function $g(z)=\displaystyle \int_0^z\left(\frac{1+t}{1-t}\right)^\lambda \frac{1}{1-t^2}dt$, $0\leq \lambda\leq 1$, in the integral representation given by \eqref{rep_int}. Taking $\lambda$ real in $[0,1]$ is without loss of generality: since $M_\alpha$ is invariant under $f(z)\mapsto e^{-i\theta}f(e^{i\theta}z)$ and $|a_3-a_2^2|$ is unchanged under this rotation, a general extremal parameter $\lambda\in\overline{\mathbb D}$ can always be rotated to a real value in $[0,1]$. For $\alpha<0$ we do not claim sharpness of \eqref{fs-ineq}; we verified the inequality itself independently of the derivation below by Monte Carlo sampling ($200\,000$ trials of $(\varepsilon_1,\varepsilon_2)$ satisfying the Carath\'eodory--Schur constraint $|\varepsilon_2|\leq1-|\varepsilon_1|^2$, with $\alpha\in[-3,3]$ away from the singular points) and by an adversarial Nelder--Mead search maximizing its violation on the extremal boundary $|\varepsilon_2|=1-|\varepsilon_1|^2$ for fifteen values of $\alpha\in[-2.9,2.9]$; in both cases no violation beyond floating-point noise ($\sim10^{-14}$) was found.
\end{remark}

\begin{proof}
Since $M_\alpha\subseteq S^*$ for every $\alpha\in\mathbb R$ (by \cite{MMR1973}), each $f\in M_\alpha$ is starlike, so we have $z\dfrac{f'}{f}(z)=\dfrac{1+\omega}{1-\omega}$ with $\omega:\D\to \D$ analytic in $\D$ and $\omega(0)=0$. Write $\omega(z)=z\omega_0(z)$. Then
\begin{equation}\label{preSch}
\frac{f''}{f'}=\frac{2\omega'}{1-\omega^2}+\frac{2\omega_0}{1-\omega}.
\end{equation}
Since $f \in M_\alpha$, there exists $\varepsilon:\D\to \D$ analytic in $\D$ with $\varepsilon(0)=0$ such that $J_\alpha[f](z)=(1+\varepsilon(z))/(1-\varepsilon(z))$, from which we get
\begin{equation}\label{omega_epsilon}
\frac{1+\omega}{1-\omega}+\alpha \frac{2z\omega'}{1-\omega^2}=\frac{1+\varepsilon}{1-\varepsilon}.
\end{equation}
Since $\varepsilon(0)=0$, it follows that $\varepsilon(z)=z\varepsilon_0(z)$. Thus
\begin{equation*}
\frac{\omega_0}{1-\omega}+\frac{\alpha\omega'}{1-\omega^2}=\frac{\varepsilon_0}{1-\varepsilon},
\end{equation*}
and using \eqref{omega_epsilon} we have
\begin{equation}\label{alpha-omega}
\frac{\alpha \omega'}{1-\omega^2} = \frac{\varepsilon_0-\omega_0}{(1-\varepsilon)(1-\omega)}.
\end{equation}
But $\omega'(z)=\omega_0(z)+z\omega_0'(z)$ and $\omega''(z)=2\omega_0'(z)+z\omega_0''(z)$. Thus, using \eqref{preSch}, we have
\[
a_3-a_2^2=\frac{1}{2}\omega''(0)-\omega'(0)^2,
\]
and by \eqref{alpha-omega} we have that $\alpha\omega'(0)=\varepsilon'(0)-\omega'(0)$, from which
\[
\omega'(0)=\frac{\varepsilon'(0)}{1+\alpha}.
\]
Additionally, differentiating \eqref{alpha-omega}, we obtain
\[
\alpha\frac{\omega''(1-\omega^2)+2\omega\omega'^2}{(1-\omega^2)^2}=\frac{(\varepsilon_0'-\omega_0')(1-\varepsilon)(1-\omega)+(\varepsilon_0-\omega_0)\left(\varepsilon'(1-\omega)+\omega'(1-\varepsilon)\right)}{(1-\varepsilon)^2(1-\omega)^2},
\]
which, evaluated at $z=0$, yields
\begin{equation}\label{ec2}
\alpha\omega''(0) = \varepsilon'(0)-\omega'(0)+(\varepsilon_0-\omega_0)(\varepsilon'(0)+\omega'(0)).
\end{equation}
Since $\varepsilon''(0)=2\varepsilon_0'(0)$ and $\omega''(0)=2\omega_0'(0)$, using these relationships in \eqref{ec2}, we get
\[
\omega''(0)=\frac{\varepsilon''(0)}{1+2\alpha}+\frac{2\alpha(2+\alpha)\varepsilon'(0)^2}{(1+\alpha)^2(1+2\alpha)}.
\]
Substituting, we have
\begin{align*}
a_3-a_2^2 & = \frac{1}{2}\left[\frac{\varepsilon''(0)}{1+2\alpha}+\frac{2\alpha(2+\alpha)\varepsilon'(0)^2}{(1+\alpha)^2(1+2\alpha)}-\frac{2\varepsilon'(0)^2}{(1+\alpha)^2}\right]\\
& = \frac{1}{2(1+2\alpha)}\left[\varepsilon''(0)-\frac{2(1-\alpha^2)\varepsilon'(0)^2}{(1+\alpha)^2}\right].
\end{align*}
If we write $\varepsilon(z)=\varepsilon_1z+\varepsilon_2z^2+\cdots=z(\varepsilon_1+\varepsilon_2z+\cdots)$ and use the fact that $|\varepsilon_2|\leq1-|\varepsilon_1|^2$, we have $|\varepsilon''(0)|\leq 2-2|\varepsilon'(0)|^2$, and with this, taking absolute values directly (so that the argument applies whatever the sign of $1+2\alpha$),
\begin{align*}
|a_3-a_2^2| & \leq \frac{1}{2|1+2\alpha|}\left[2-2|\varepsilon'(0)|^2+\frac{2|1-\alpha^2||\varepsilon'(0)|^2}{(1+\alpha)^2}\right]\\
& = \frac{1}{|1+2\alpha|}\left[1+\left(\left|\frac{1-\alpha}{1+\alpha}\right|-1\right)|\varepsilon'(0)|^2\right].
\end{align*}
Since $\varepsilon'(0)=(1+\alpha)\omega'(0)$ and $f''(0)=4\omega'(0)$ (see \eqref{preSch}), we conclude that
\[
|a_3-a_2^2|\leq \frac{1}{|1+2\alpha|}\left[1+\left(\frac{|1-\alpha^2|-|1+\alpha|^2}{4}\right)|a_2|^2\right].
\]
For $\alpha>0$, equality holds when $\varepsilon_0(z)=\dfrac{z+\lambda}{1+\lambda z}$, $0\leq \lambda\leq 1$, which, in the integral representation \eqref{rep_int}, corresponds to a function $g$ given by $g'(z)=\left(\dfrac{1+z}{1-z}\right)^\lambda \cdot \dfrac{1}{1-z^2}$.
\end{proof}

\subsection{Logarithmic coefficients}

For $f\in S$, $f(z)=z+a_2z^2+a_3z^3+\cdots$, the logarithmic coefficients $\gamma_n=\gamma_n(f)$ are defined by
\[
\log\frac{f(z)}{z}=2\sum_{n=1}^\infty \gamma_n z^n,\qquad z\in\D,
\]
so that $\gamma_1=a_2/2$ and $\gamma_2=a_3/2-a_2^2/4$. These quantities have attracted considerable recent attention in geometric function theory, in connection with Milin's conjecture and its role in de Branges' proof of the Bieberbach conjecture; sharp bounds are known, for instance, for $S^\ast$ ($|\gamma_n|\leq 1/n$, see \cite[Thm.~8.4]{Dur1983}) and for $C$ ($|\gamma_1|\leq1/2$, $|\gamma_2|\leq1/4$, see \cite{PSW2020}). We determine the sharp bounds for $\gamma_1,\gamma_2$ on the full class $M_\alpha$.

\begin{theorem}\label{thm:logcoeff}
Let $\alpha\in\mathbb R\setminus\{-1,-1/2\}$ and $f\in M_\alpha$. Then
\begin{equation}\label{gamma1-bound}
|\gamma_1|\leq \frac{1}{|1+\alpha|},
\end{equation}
\begin{equation}\label{gamma2-bound}
|\gamma_2|\leq \max\left\{\frac{1}{2|1+2\alpha|},\ \frac{|\alpha^2+4\alpha+1|}{2(1+\alpha)^2|1+2\alpha|}\right\}.
\end{equation}
For $\alpha>0$, both bounds are sharp.
\end{theorem}

\begin{proof}
From the proof of the Fekete--Szeg\H o theorem, $a_2=2\varepsilon'(0)/(1+\alpha)$, so
\[
\gamma_1=\frac{a_2}{2}=\frac{\varepsilon'(0)}{1+\alpha},
\]
and $|\gamma_1|\leq 1/|1+\alpha|$ follows from $|\varepsilon'(0)|\leq1$. Next, using $a_3=(a_3-a_2^2)+a_2^2$ and the exact identity for $a_3-a_2^2$ established in that proof,
\[
\gamma_2=\frac{a_3}2-\frac{a_2^2}4=\frac{a_2^2}4+\frac1{4(1+2\alpha)}\left[\varepsilon''(0)-\frac{2(1-\alpha^2)\varepsilon'(0)^2}{(1+\alpha)^2}\right].
\]
Substituting $a_2^2=4\varepsilon'(0)^2/(1+\alpha)^2$ and simplifying,
\[
\gamma_2=\frac{\alpha^2+4\alpha+1}{2(1+\alpha)^2(1+2\alpha)}\,\varepsilon'(0)^2+\frac{\varepsilon''(0)}{4(1+2\alpha)}.
\]
Write $x=|\varepsilon'(0)|^2\in[0,1]$. By $|\varepsilon''(0)|\leq2-2|\varepsilon'(0)|^2=2(1-x)$ and the triangle inequality, with the phase of $\varepsilon''(0)$ chosen (independently of $\varepsilon'(0)$, since $\varepsilon_2$ is a free coefficient subject only to the modulus bound $|\varepsilon_2|\leq1-|\varepsilon_1|^2$) to align both terms,
\[
|\gamma_2|\leq \frac{|\alpha^2+4\alpha+1|}{2(1+\alpha)^2|1+2\alpha|}\,x+\frac{1-x}{2|1+2\alpha|}=:g(x),
\]
and this bound is attained at each fixed $x$. Since $g$ is affine in $x$, its supremum over $x\in[0,1]$ is attained at an endpoint, giving $\max\{g(0),g(1)\}$, which is \eqref{gamma2-bound}.

For $\alpha>0$, the extremal family $\varepsilon_0(z)=(z+\lambda)/(1+\lambda z)$, $0\leq\lambda\leq1$, used in the proof of the Fekete--Szeg\H o theorem gives $\varepsilon_1=\lambda$, $\varepsilon_2=1-\lambda^2$, i.e.\ $x=\lambda^2$ ranges over all of $[0,1]$; thus both endpoints $x=0$ (at $\lambda=0$) and $x=1$ (at $\lambda=1$) are attained by genuine functions in $M_\alpha$, via the same integral representation \eqref{rep_int}, proving sharpness of \eqref{gamma2-bound} (and, at $\lambda=1$, of \eqref{gamma1-bound} as well).
\end{proof}

\begin{remark}
At $\alpha=0$ ($M_0=S^\ast$), \eqref{gamma1-bound}--\eqref{gamma2-bound} give $|\gamma_1|\leq1$, $|\gamma_2|\leq1/2$, matching \cite[Thm.~8.4]{Dur1983}; at $\alpha=1$ ($M_1=C$), they give $|\gamma_1|\leq1/2$, $|\gamma_2|\leq1/4$, matching \cite{PSW2020}. We verified \eqref{gamma2-bound} independently of the derivation above by evaluating $\gamma_2$ exactly (as a function of $\varepsilon_1,\varepsilon_2,\alpha$) over $300\,000$ Monte Carlo samples satisfying the Carath\'eodory--Schur constraint, with $\alpha\in[-3,3]$ away from the singular points: no violation was found, and a fine boundary search ($|\varepsilon_2|=1-|\varepsilon_1|^2$) reproduced the right-hand side of \eqref{gamma2-bound} to within $10^{-6}$ at every tested $\alpha$, including negative values, which supports (but, absent an explicit extremal function for $\alpha<0$, does not prove) sharpness there as well.
\end{remark}

\subsection{The order of a function}

The \emph{order} of a locally univalent mapping $f$ is defined as the supremum of the modulus of the second Taylor coefficient of the function
\[
\frac{f(\sigma_a(z)) - f(a)}{f'(a)(1 - |a|^2)},
\]
where $\sigma_a(z) = (z + a)/(1 + \overline{a}z)$ and $a \in \D$. This invariant, denoted $A_f$, is given by
\begin{equation*}
A_f=\sup_{z\in\mathbb{D}}\left|\frac{(1-|z|^2)}{2}\frac{f''}{f'}(z)-\overline{z}\right|.
\end{equation*}
By Pommerenke's classical paper \cite{P1964}, $A_f\geq 1$ for any $f$, with $A_f=1$ if and only if $f$ is a convex mapping, and $A_f\leq 2$ whenever $f$ is univalent.

\begin{theorem}
Let $f\in M_\alpha$ with $\alpha\in [0,1]$. Then its order satisfies $A_f\leq 2-\alpha$.
\end{theorem}

\begin{proof}
Subtracting $1$ from both sides of \eqref{omega_epsilon} we obtain
\begin{equation*}
\frac{2\omega}{1-\omega}+\frac{2\alpha z\omega'}{1-\omega^2}=\frac{2\varepsilon}{1-\varepsilon},
\end{equation*}
and, dividing by $z$ (recall $\omega=z\omega_0$, $\varepsilon=z\varepsilon_0$),
\begin{equation*}
\frac{2\omega_0}{1-\omega}=\frac{2\varepsilon_0}{1-\varepsilon}-\frac{2\alpha\omega'}{1-\omega^2}.
\end{equation*}
Replacing the right-hand side of the last equation in \eqref{preSch}, we have
\begin{equation*}
\frac{f''}{f'}=\frac{2\omega'}{1-\omega^2}+\frac{2\varepsilon_0}{1-\varepsilon}-\frac{2\alpha\omega'}{1-\omega^2}=\frac{2\varepsilon_0}{1-\varepsilon}+\frac{2(1-\alpha)\omega'}{1-\omega^2}.
\end{equation*}
Then, with $\varepsilon_0$, $\omega$ and $\omega'$ evaluated at $z$,
\begin{align*}
\frac{1}{2}(1-|z|^2)\frac{f''}{f'}(z)-\overline{z} & = (1-|z|^2)\left(\frac{\varepsilon_0}{1-\varepsilon}+(1-\alpha)\frac{\omega'}{1-\omega^2}\right)-\overline{z}\\
& = \frac{\varepsilon_0-\overline{z}}{1-z\varepsilon_0}+(1-\alpha)\frac{\omega'(1-|z|^2)}{1-\omega^2}.
\end{align*}
Applying the Schwarz--Pick lemma and the fact that $|(\varepsilon_0-\overline{z})/(1-z\varepsilon_0)|<1$ for all $|z|<1$, we obtain
\begin{align*}
\left|\frac{1}{2}(1-|z|^2)\frac{f''}{f}(z)-\overline{z}\right| & \leq \left|\frac{\varepsilon_0-\overline{z}}{1-z\varepsilon_0}\right|+(1-\alpha)\left|\frac{\omega'(1-|z|^2)}{1-\omega^2}\right|\\
& \leq 1+(1-\alpha)(1-|z|^2)\frac{|\omega'|}{1-|\omega|^2}\\
& \leq 1+(1-\alpha)\frac{1-|z|^2}{1-|\omega|^2}\cdot\frac{1-|\omega|^2}{1-|z|^2}=1+(1-\alpha)=2-\alpha.
\end{align*}
Taking the supremum, we have $A_f\leq 2-\alpha$, which completes the proof.
\end{proof}

\begin{remark}
The proof only uses $1-\alpha\geq0$, so it formally extends to every $\alpha\leq1$, including $\alpha<0$. This extension carries no information, however: for $\alpha<0$ it gives $A_f\leq 2-\alpha>2$, which is weaker than the universal bound $A_f\leq2$ recalled above for every univalent function, with no reference to $M_\alpha$.
\end{remark}

\subsection{Schwarzian derivative}

The Schwarzian derivative of a locally univalent mapping $f$ is
\[
S_f(z) = \left( \frac{f''(z)}{f'(z)} \right)' - \frac{1}{2} \left( \frac{f''(z)}{f'(z)} \right)^2.
\]
This is one of the most thoroughly studied invariants in complex analysis; see \cite{CDO2011} and the references therein. Since every function in $M_\alpha$ is univalent, its Schwarzian norm satisfies
\[
\|Sf\| = \sup_{z \in \mathbb{D}} |Sf(z)|(1-|z|^2)^2 \leq 6.
\]
We seek a sharper bound in terms of $\alpha$. As mentioned in the introduction, if $\alpha \geq 1$, then $M_\alpha \subseteq S^\ast_{\delta(\alpha)}$ and $M_\alpha\subseteq C_\beta$, where $\beta=(1-1/\alpha)\delta(\alpha)$ and $\delta(\alpha)$ is the increasing function recalled in \eqref{delta-formula}. By \cite{MMR1978}, $0\leq \delta(\alpha)\leq 1$ for all $\alpha\geq1$.

\begin{theorem}\label{thm:schwarzian-bound}
If $\alpha\geq 0$ and $f\in M_\alpha$, then $\|Sf\|\leq s(\alpha)$, where
\begin{equation}\label{def-s}
s(\alpha)=\begin{cases}
6, & 0\leq \alpha\leq \alpha_0,\\[0.3cm]
\dfrac2\alpha\left(2-\alpha^2\right), & \alpha_0<\alpha<1,\\[0.3cm]
2, & 1\leq \alpha\leq \alpha_1,\\[0.3cm]
8\left(1-\dfrac1\alpha\right)\delta(\alpha)\left(1-\left(1-\dfrac1\alpha\right)\delta(\alpha)\right), & \alpha>\alpha_1.
\end{cases}
\end{equation}
Here $\alpha_0=(\sqrt{17}-3)/2\sim 0.56$ and $\alpha_1$ is the number satisfying $\left(1-\dfrac{1}{\alpha_1}\right)\delta(\alpha_1)=1/2$ ($\alpha_1\sim 3.27$).
\end{theorem}

\begin{proof}
For $\alpha\in[0,1]$ we have $f\in S^\ast$, so $zf'(z)/f(z)=(1+\omega(z))/(1-\omega(z))$ and
\begin{equation}\label{conv-omega}
1+z\frac{f''}{f'}(z)=\frac{z\omega'(z)}{1+\omega(z)}+\frac{z\omega'(z)}{1-\omega(z)}+\frac{1+\omega(z)}{1-\omega(z)}.
\end{equation}
Also $J_\alpha[f](z)=(1+\varepsilon(z))/(1-\varepsilon(z))$, therefore
\begin{equation}\label{conv-eps}
1+z\frac{f''}{f'}(z)=\frac1\alpha\cdot \frac{1+\varepsilon(z)}{1-\varepsilon(z)}-\frac{1-\alpha}{\alpha}\cdot\frac{1+\omega(z)}{1-\omega(z)}.
\end{equation}
Recall that $\varepsilon(z)=z\varepsilon_0(z)$ and $\omega(z)=z\omega_0(z)$ as before. Thus, omitting the argument $(z)$,
\[
\frac{f''}{f'}=\frac2\alpha\left(\frac{\varepsilon_0}{1-\varepsilon}-\frac{(1-\alpha)\omega_0}{1-\omega}\right).
\]
Then
\[
Sf=\frac{2}{\alpha}\left(\frac{\varepsilon_0'+\varepsilon_0^2}{(1-\varepsilon)^2}-(1-\alpha)\frac{\omega_0'+\omega_0^2}{(1-\omega)^2}\right)-\frac{2}{\alpha^2}\left(\frac{\varepsilon_0}{1-\varepsilon}-\frac{(1-\alpha)\omega_0}{1-\omega}\right)^2,
\]
which is equivalent to
\[
Sf=\frac{2}{\alpha}\left(\frac{\varepsilon_0'}{(1-\varepsilon)^2}-(1-\alpha)\frac{\omega_0'}{(1-\omega)^2}\right)-\frac{2(1-\alpha)}{\alpha^2}\left(\frac{\varepsilon_0}{1-\varepsilon}-\frac{\omega_0}{1-\omega}\right)^2.
\]
Using \eqref{omega_epsilon}, we have
\[
Sf=\frac{2}{\alpha}\left(\frac{\varepsilon_0'}{(1-\varepsilon)^2}-(1-\alpha)\frac{\omega_0'}{(1-\omega)^2}\right)-2(1-\alpha)\left(\frac{\omega'}{1-\omega^2}\right)^2.
\]
Using the Schwarz--Pick lemma and $1-|\varphi_g(z)|^2=(1-|g(z)|^2)(1-|z|^2)/|1-zg(z)|^2$, where $\varphi_g(z)=(\overline z-g(z))/(1-zg(z))$, we have
\begin{equation}\label{bound-sf}
\begin{aligned}
|Sf(z)|(1-|z|^2)^2 &\leq \frac{2}{\alpha}\left(1-|\varphi_{\varepsilon_0}(z)|^2+(1-\alpha)(1-|\varphi_{\omega_0}(z)|^2)\right)\\
&\quad+2(1-\alpha)\left(\frac{|\omega'|(1-|z|^2)}{1-|\omega|^2}\right)^2\\
&\leq \frac{2}{\alpha}\left(1+1-\alpha\right)+2(1-\alpha)= \frac{2}{\alpha}\left(2-\alpha^2\right).
\end{aligned}
\end{equation}
Note that $|Sf(z)|(1-|z|^2)^2 \leq 6$ whenever $\alpha \geq (\sqrt{17}-3)/2 \approx 0.56$, so the first two parts of $s(\alpha)$ are proved. On the other hand, when $\alpha\geq 1$, we have $f\in C_\beta$, with $\beta=(1-1/\alpha)\delta(\alpha)$. Suita \cite{Sui1996} proved that, for any $f\in C_\beta$, the norm of the Schwarzian derivative is bounded by $2$ when $0\leq \beta\leq 1/2$, and by
\[
8\beta(1-\beta)=8\left(1-\frac1\alpha\right)\delta(\alpha)\left(1-\left(1-\frac1\alpha\right)\delta(\alpha)\right)
\]
when $1/2<\beta<1$. Since $\beta=1/2$ occurs when $\left(1-\frac{1}{\alpha_1}\right)\delta(\alpha_1)=1/2$, the proof is completed.
\end{proof}

\begin{corol}
For $0<\alpha<1$, we have
\begin{equation*}
|Sf(z)|(1-|z|^2)^2+ \frac{2}{\alpha}\left|\frac{\overline z-\varepsilon_0(z)}{1-\varepsilon(z)}\right|^2+(1-\alpha)\left|\frac{\overline z-\omega_0(z)}{1-\omega(z)}\right|^2\leq \frac{2}{\alpha}(2-\alpha^2).
\end{equation*}
\end{corol}

\begin{proof}
In the first line of \eqref{bound-sf}, the three terms
\[
1-|\varphi_{\varepsilon_0}(z)|^2,\qquad 1-|\varphi_{\omega_0}(z)|^2,\qquad \left(\frac{|\omega'|(1-|z|^2)}{1-|\omega|^2}\right)^2
\]
were bounded above by $1$, $1$ and $1$, respectively (via the Schwarz--Pick lemma), to reach the second line of \eqref{bound-sf}. Keeping these three non-negative slack terms on the left-hand side instead of discarding them, and using
\[
\varphi_{\varepsilon_0}(z)=\frac{\overline z-\varepsilon_0(z)}{1-\varepsilon(z)},\qquad \varphi_{\omega_0}(z)=\frac{\overline z-\omega_0(z)}{1-\omega(z)},
\]
the first inequality in \eqref{bound-sf} rearranges directly to the stated inequality.
\end{proof}

As an application of Theorem~\ref{thm:schwarzian-bound}, we record a quasiconformal extension result for the members of $M_\alpha$ with large $\alpha$.

\begin{corol}\label{cor:qc-ext}
Let $\alpha>\alpha_1$. Since $\delta(\alpha)$ and $1-1/\alpha$ are both increasing and non-negative for $\alpha\geq 1$, the product $\beta(\alpha):=(1-1/\alpha)\delta(\alpha)$ is increasing on $[1,\infty)$, and $\beta(\alpha_1)=1/2$ by the defining property of $\alpha_1$. Hence $\beta(\alpha)>1/2$ for every $\alpha>\alpha_1$, so that, in the fourth branch of \eqref{def-s},
\[
s(\alpha)=8\beta(\alpha)\bigl(1-\beta(\alpha)\bigr)<8\cdot\frac12\cdot\frac12=2.
\]
Consequently, every $f\in M_\alpha$ with $\alpha>\alpha_1$ satisfies $\|Sf\|\leq s(\alpha)<2$, and the Ahlfors--Weill extension theorem \cite{AW1962} applies: $f$ admits an explicit quasiconformal extension to $\widehat{\mathbb{C}}$, whose maximal dilatation does not exceed $s(\alpha)/2<1$.
\end{corol}

\begin{remark}
Unlike Theorem~\ref{thm:main} and Theorem~\ref{thm:fs}, the bound $s(\alpha)$ of Theorem~\ref{thm:schwarzian-bound} does not extend usefully to $\alpha<0$. The step from the first to the second line of \eqref{bound-sf} factors $2/\alpha$ out of a sum of non-negative quantities and bounds that sum above, which reverses the inequality once $2/\alpha<0$. Replacing it with a direct triangle-inequality bound on the (sign-free) identity preceding \eqref{bound-sf} gives, for $\alpha<0$, $\|Sf\|\leq s_-(\alpha):=-4/\alpha-2\alpha+4$. This function attains its minimum over $\alpha<0$ at $\alpha=-\sqrt2$, where $s_-(-\sqrt2)=4\sqrt2+4\approx9.657$, already above the universal bound $\|Sf\|\leq6$. Hence this extension --- and with it the quasiconformal-extension Corollary~\ref{cor:qc-ext}, whose hypothesis $s(\alpha)<2$ has no counterpart here --- carries no information for any $\alpha<0$; a genuine bound in that range, if one exists, would require a different argument.
\end{remark}

\subsubsection*{The Schwarzian derivative of the $\alpha$-Koebe function}

We end this work with an example associated with the integral representation of any given function in $M_\alpha$ for $\alpha>0$. In \cite{MMR1974} the authors show that $f\in M_\alpha$ ($\alpha>0$) if and only if there exists $g\in S^*$ such that
\begin{equation}
f(z)=\left[\frac{1}{\alpha}\int_0^z\frac{(g(\xi))^{1/\alpha}}{\xi}\,d\xi\right]^\alpha.
\end{equation}
This integral representation allows us to construct $\alpha$-convex functions from any starlike function. In particular, if $k(z)=z/(1-z)^2$ ($z\in \D$) is the Koebe function, we obtain the $\alpha$-convex Koebe function
\begin{equation}
k_\alpha(z)=\left[\frac{1}{\alpha}\int_0^z \xi^{1/\alpha-1}(1-\xi)^{-2/\alpha}\,d\xi\right]^{\alpha}.
\end{equation}
It is well known that the Koebe function maximizes many important geometric properties in certain classes of univalent functions; in particular, $k_\alpha$ yields sharp growth estimates in $M_\alpha$ \cite{M1973}.

From an algebraic point of view, in \eqref{rep_int} let $G(z)=\int_0^z(g(\xi))^{1/\alpha}\cdot\xi^{-1}\,d\xi$, so that $f(z)=[(1/\alpha)\cdot G(z)]^\alpha$. Then
\begin{equation*}
\frac{f''}{f'}(z)=(\alpha-1)\frac{G'}{G}(z)+\frac{G''}{G'}(z).
\end{equation*}
Thus we obtain an expression for the Schwarzian derivative in terms of $G$:
\begin{equation}\label{schwarz-G}
Sf(z)=\left(\frac{f''}{f'}(z)\right)'-\frac{1}{2}\left(\frac{f''}{f'}(z)\right)^2=SG(z)+\frac{1-\alpha^2}{2}\left(\frac{G'}{G}(z)\right)^2.
\end{equation}
Notice that $G$ is not necessarily locally univalent in $\D$. Now, if we let $n=1/\alpha$ for $n\in\mathbb{N}$ with $n\geq 2$, \eqref{koebeconvex} can be expressed as
\begin{equation}\label{koebeconvex2}
k_n(z)=\left[n\int_0^z \xi^{n-1}(1-\xi)^{-2n}\,d\xi\right]^{1/n}=[n\operatorname{B}(z;n,1-2n)]^{1/n},
\end{equation}
where $B(z;a,b)$ is the incomplete beta function. Since the function in the integrand is rational, \eqref{koebeconvex2} can also be represented as
\begin{equation*}
k_n(z)=\left[n\sum_{k=0}^{n-1}(-1)^k\binom{n-1}{k}\left(\frac{1-(1-z)^{k-2n+1}}{k-2n+1}\right)\right]^{1/n}.
\end{equation*}
In this case $G(z)=B[z;n,1-2n]$ and, applying \eqref{schwarz-G} with $\alpha=1/n$, we obtain
\begin{equation}\label{Schwarz-k}
Sk_n(z)=\frac{1-n}{2z^2}\left[1+n\left(\frac{1+z}{1-z}\right)^2\right]+\left(\frac{n^2-1}{2n^2}\right)\left(\frac{G'}{G}(z)\right)^2.
\end{equation}
Evaluating \eqref{Schwarz-k} at $z=0$ (by L'H\^opital, since both terms have a removable singularity there) gives $Sk_n(0)=6n(n-1)/((n+2)(n+1))$, so
\[
\sup_{f\in M_\alpha}\|Sf\|\geq \frac{6n(n-1)}{(n+2)(n+1)}.
\]
For $\alpha = 1/n$, the expression $Sk_n$ is a rational function of $z$ and $G(z)$. This allows us to compute the Schwarzian norm numerically for each $n$. For $n = 1, \ldots, 9$, a direct computation (using Mathematica) yields
\begin{equation}\label{schwarz-k-n}
\|Sk_n\|=\sup_{z\in\D}|Sk_n(z)|(1-|z|^2)^2=2\left(3-\frac1n\right)\left(1-\frac1n\right), \qquad n=1,\ldots, 9.
\end{equation}
Numerical evidence for $10 \leq n < 50$ is consistent with this formula, as Figure~\ref{fig1} illustrates. Thus, we conjecture that
\[
\sup_{f\in M_\alpha}\|Sf\|=2\left(3-\frac1n\right)\left(1-\frac1n\right).
\]

\begin{figure}[h]
\centering
\includegraphics[width=\textwidth]{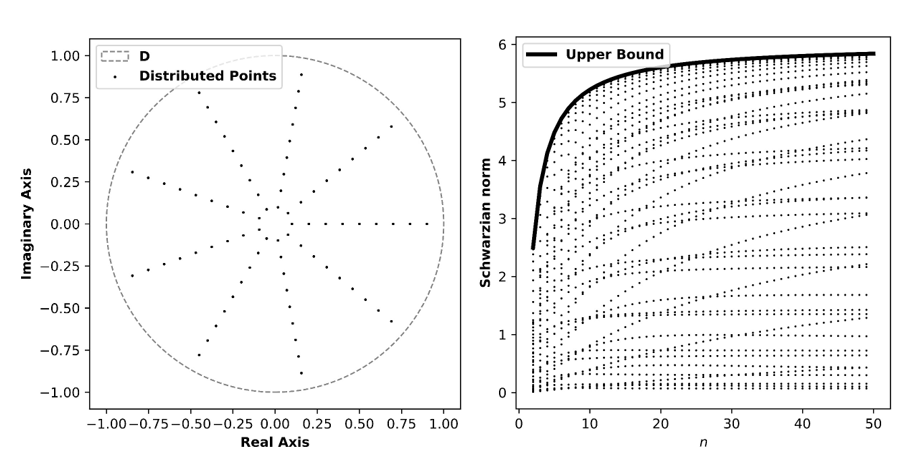}
\caption{(Left) Points in $\D$. (Right) $|Sk_n(z)|(1-|z|^2)^2$.}
\label{fig1}
\end{figure}

\subsubsection*{The second Hankel determinant}

For $f(z)=z+a_2z^2+a_3z^3+a_4z^4+\cdots$, the second Hankel determinant is $H_2(2)=|a_2a_4-a_3^2|$. It is known that $H_2(2)\leq1$ for $f\in S^\ast$ and $H_2(2)\leq1/8$ for $f\in C$, both sharp \cite{JHD2007}. We record here a conjectural sharp bound for the full class $M_\alpha$, obtained as follows.

Differentiating \eqref{alpha-omega} once more and proceeding exactly as in the proof of the Fekete--Szeg\H o theorem, but one order further, gives $a_4$ as an explicit (computer-algebra-verified) rational function of $\varepsilon_1:=\varepsilon'(0)$, $\varepsilon_2:=\varepsilon''(0)/2$, $\varepsilon_3:=\varepsilon'''(0)/6$ and $\alpha$, singular only at $\alpha\in\{-1,-1/2,-1/3\}$. The Taylor coefficients $(\varepsilon_1,\varepsilon_2,\varepsilon_3)$ of the bounded function $\varepsilon_0=\varepsilon(z)/z:\D\to\overline{\D}$ are not independently constrained by $|\varepsilon_2|\leq1-|\varepsilon_1|^2$, $|\varepsilon_3|\leq1-|\varepsilon_2|^2$ --- unlike the pair $(\varepsilon_1,\varepsilon_2)$ used in the Fekete--Szeg\H o theorem, for which the analogous single inequality is exact. The correct joint region is parametrized (via the classical Schur algorithm) by
\[
\varepsilon_1=p,\qquad \varepsilon_2=(1-|p|^2)x,\qquad \varepsilon_3=(1-|p|^2)\bigl[(1-|x|^2)y-\overline p\,x^2\bigr],\qquad |p|,|x|,|y|\leq1.
\]
Substituting this parametrization into $a_2a_4-a_3^2$ and restricting to real $p\in[-1,1]$, $x=\pm1$ --- where an extensive numerical search (constrained global optimization over $p,x,y$ in the closed unit ball, for more than thirty values of $\alpha\in[-4,6]$) locates the supremum to within machine precision --- yields, with $u=p^2$,
\[
Q_\pm(\alpha):=\max_{u\in[0,1]}\left|\frac{N_\pm(\alpha,u)}{3(\alpha+1)^3(2\alpha+1)^2(3\alpha+1)}\right|,
\]
\begin{align*}
N_+(\alpha,u)&=-4(13\alpha^2+10\alpha+1)u^2+4(\alpha+1)^2(4\alpha^2+7\alpha+1)u-3(\alpha+1)^3(3\alpha+1),\\
N_-(\alpha,u)&=4\alpha(7\alpha^3+24\alpha^2+15\alpha+2)u^2-12\alpha(\alpha+1)^3u-3(\alpha+1)^3(3\alpha+1).
\end{align*}

\textbf{Conjecture.} For $\alpha\in\mathbb R\setminus\{-1,-1/2,-1/3\}$,
\[
\sup_{f\in M_\alpha}|a_2a_4-a_3^2|=\max\{Q_+(\alpha),\,Q_-(\alpha)\}.
\]

This is consistent with $\alpha=0$ and $\alpha=1$, where it gives exactly $Q_+(0)=Q_-(0)=1$ and $\max\{Q_+(1),Q_-(1)\}=\max\{1/9,1/8\}=1/8$, matching \cite{JHD2007} in both cases. Table~\ref{tab:hankel} lists $\max\{Q_+(\alpha),Q_-(\alpha)\}$ on an integer grid $\alpha=-3,\ldots,3$, omitting only the singular point $\alpha=-1$; the values are not symmetric about $\alpha=0$ because all three singularities $-1,-1/2,-1/3$ lie on the negative axis, so the bound grows without an analogous blow-up for $\alpha>0$.

\begin{table}[h]
\centering
\begin{tabular}{c|cccccc}
$\alpha$ & $-3$ & $-2$ & $0$ & $1$ & $2$ & $3$\\
\hline
$\max\{Q_+,Q_-\}$ & $0.0625$ & $1$ & $1$ & $0.125$ & $0.0450$ & $0.0230$
\end{tabular}
\caption{Conjectural sharp bound for $H_2(2)=|a_2a_4-a_3^2|$ on $M_\alpha$, on an integer grid avoiding the singular point $\alpha=-1$.}
\label{tab:hankel}
\end{table}

We emphasize that this conjecture records the exact value on the two-parameter family $\{(p,\pm1,y):p\in[-1,1],\,|y|\leq1\}$, together with strong numerical evidence (not a proof) that this family contains the global extremizer; establishing this in general appears to require the kind of case-by-case real-variable analysis found in dedicated papers on second Hankel determinants, which we have not carried out.

\end{document}